\documentclass[reqno]{amsart}
\usepackage{latexsym,amssymb,amsthm,amsmath}

\usepackage{amssymb}
\usepackage{xcolor}
\usepackage{amsmath,hyperref}
\hypersetup{colorlinks=true,
	 linkcolor=blue,
 filecolor=magenta,
urlcolor=cyan,}
\theoremstyle{plain}
\newtheorem{theorem}{Theorem}[section]
\newtheorem{lemma}{Lemma}[section]

\newtheorem{corollary}{Corollary}[section]

\numberwithin{equation}{section}
\newtheorem {conjecture}{Conjecture}
\theoremstyle{remark}
\newtheorem{remark}{Remark}[section]
\allowdisplaybreaks
\begin{document}

\title[Gradient estimates for Lane-Emden systems on RCD*(-K,N) spaces]
{An asymptotic Liouville-type theorem for systems of Lane-Emden equations on RCD*(-K,N) spaces via gradient estimates}

\author[S. Bhattacharyya]{Sujit Bhattacharyya}

\subjclass[2020]{35J47, 35B53, 35R01}
\keywords{Elliptic gradient estimation; metric measure spaces; Lane-Emden equation; Liouville type theorem;}

\begin{abstract}
In this article, we establish elliptic gradient estimates for positive solutions of the Lane-Emden system on metric measure spaces satisfying the synthetic Ricci curvature-dimension condition $\mathrm{RCD}^*(-K,N)$. Our approach combines the weak differential calculus and the Bochner inequality available in the $\mathrm{RCD}^*(-K,N)$ setting with suitable auxiliary function arguments, extending classical gradient estimate techniques to nonsmooth spaces.
The resulting estimate is valid for all exponents $\alpha,\beta>0$ and all $N\in[1,\infty)$, with no dimensional case distinction and with explicit constants. As an application, we prove a quantitative, asymptotic Liouville-type theorem on $\mathrm{RCD}^*(0,N)$ spaces: for positive solutions that are bounded above and below by fixed positive constants, the normalized gradients $|\nabla f|/f$ and $|\nabla h|/h$ are bounded by a quantity that is of order $\max\{\alpha,\beta\}$ and hence tends to zero as the exponents tend to zero (with comparable size). We also explain why the method gives no rigidity for large exponents, which is consistent with, but does not prove, the existence of nonconstant solutions there.
These results generalize corresponding results from the smooth Riemannian setting and contribute to the study of nonlinear elliptic systems on spaces with synthetic Ricci curvature lower bounds.
\end{abstract}
\maketitle
 
\section{Introduction} 
The relation between the geometry of an underlying space and the qualitative behavior of solutions to elliptic equations has been a central theme in modern geometric analysis. In recent years, the development of analysis on metric measure spaces satisfying synthetic lower Ricci curvature bounds has significantly broadened the scope of classical elliptic theory beyond the smooth Riemannian setting. In particular, $\mathrm{RCD}^*(K,N)$ spaces provide a natural framework in which many fundamental analytical tools-including Sobolev calculus, heat flow, and Bochner-type inequalities remain available despite the possible presence of singularities.

Among various elliptic partial differential equations, the Lane-Emden equation occupies a distinguished position because of its rich mathematical structure and its connections with astrophysics, nonlinear potential theory, and reaction-diffusion phenomena. While regularity and gradient estimates for solutions have been extensively studied on smooth manifolds, comparatively little is known in the synthetic setting of $\mathrm{RCD}^*(K,N)$ spaces, where the absence of a differentiable structure requires fundamentally different analytical techniques.

The purpose of this paper is to establish gradient estimates for positive solutions of a class of Lane-Emden systems on $\mathrm{RCD}^*(K,N)$ spaces. Our approach combines the weak differential calculus available on metric measure spaces with Bochner-type identities and suitable auxiliary function arguments, extending classical gradient estimate techniques to the nonsmooth setting.

\subsection{System of Lane-Emden equations and its significance}
\textit{``How can the pressure inside a star balance its own gravitational attraction?"}-this fundamental question motivated J. H. Lane \cite{Lane1870} and R. Emden \cite{Emden1907} to study stellar structures and we got the classical Lane-Emden equation, given by
\begin{equation}
\label{eq_lane-emden1}
    \Delta_{\mathbb{R}^n} u = -u^p,\text{ on $\mathbb{R}^n$}
\end{equation}
where $p>1$ is the nonlinear constant, and $\Delta_{\mathbb{R}^n}$ is the Laplacian on the Euclidean space. This equation resolved several questions like, How density varies from the center to the surface? How pressure changes inside a star? How temperature is distributed? Why stars remain stable despite their enormous gravitational attraction? etc. This equation and its variations have been extensively studied in the Euclidean space for example, Adomian et al. \cite{Adomian1995} studied the analytic solutions of the Lane-Emden equation over $\mathbb{R}^n$. Gidas and Spruck \cite{Gidas1981} investigated this equation and proved that the equation \eqref{eq_lane-emden1} possesses no positive solution on $\mathbb{R}^n$ if $n>1$ and $1\le p<\frac{n+2}{n-2}$. This gives us idea on how dimension of the space affects the solution of elliptic equations. Gidas and Spruck's idea can be considered as an extended version of the Liouville type theorem for elliptic equations. Interested readers can see the book of Chandrasekhar \cite[pp.~89-90]{Chandrasekhar1939} for different transformations on the Lane-Emden equation. 

Coming back to recent works, the equation has been extended to system of equations in the following form
\begin{eqnarray}
\label{eq_lm-1}
\begin{cases}
    \Delta f = -h^p,\\
    \Delta h = -f^q,
\end{cases}
\text{ on $\mathbb{R}^n$},
\end{eqnarray}
and has been rigorously investigated. The main investigation on the equation \eqref{eq_lm-1} is centered on the following famous conjecture
\begin{conjecture}
    (\textbf{Lane-Emden conjecture} \cite{LiZhang2018}) If the pair $(p,q)$ is subcritical i.e., if the pair $(p,q)$ lies below the Sobolev hyperbola, $$p>0,q>0,\frac{1}{p+1}+\frac{1}{q+1}>1-\frac{2}{N},$$ then system \eqref{eq_lm-1} (here $N$ is the dimension of the ambient space $\mathbb{R}^N$) has no positive classical solutions.
\end{conjecture}
In 2002, Busca and Man\'asevich \cite{Busca2002} provided a partial positive answer to that conjecture by showing the nonexistence of positive solutions to Lane-Emden systems below the critical Sobolev hyperbola over the Euclidean space. Quaas and Xia \cite{QuaasXia2015} investigated the following system of Lane-Emden equations
\begin{eqnarray}
\label{eq_lm_2}
\begin{cases}
    (-\Delta)^\alpha f = h^p,\\
    (-\Delta h)^\alpha = f^q,
\end{cases}
\text{ on $\mathbb{R}^n$},
\end{eqnarray}
involving fractional order Laplacian over $\mathbb{R}^n$ and derived some Liouville type results. Aurther and Yan \cite{AurtherYan2002} investigated a more complex extension known as \'Henon-Lane-Emden system given by
\begin{eqnarray}
\label{eq_hlm_1}
\begin{cases}
    (-\Delta)^m f = |x|^a h^p,\\
    (-\Delta h)^m = |x|^b f^q,
\end{cases}
\text{ on $\mathbb{R}^n$}.
\end{eqnarray}
They derived Liouville type theorems by imposing certain conditions on the parameters $m,p,q,a,b$. In the year 2018, Xu et al. \cite{Xu2018} studied Liouville type theorems for Lane-Emden system of equations (in inequality form) given by
\begin{eqnarray}
\label{eq_ineq}
\begin{cases}
    \Delta_m f + h^{\sigma_1} \le 0,\\
    \Delta_m h + f^{\sigma_2} \le 0,
\end{cases}
\end{eqnarray}
on a geodesically complete noncompact connected Riemannian manifold $M$, involving $p$-Laplacian operator defined by $\Delta_p u:= \text{div}(|\nabla u|^{p-1}\nabla u)$. This article extended the analysis of Lane-Emden system to non Euclidean setting. This allows us to investigate curvature dependent gradient estimates, using which we can say at what curvature condition the solutions of the above mentioned equation will be constant. Li and Souplet \cite{LiSouplet2025} studied the Lane-Emden system on the upper half space of $\mathbb{R}^n$. This allowed us to understand how solutions behave on manifolds with boundary. Interested readers can also follow the work of He et al. \cite{HeSunWang2025} to observe how the solutions of generalized $p$-Lane-Emden equation $\Delta_p u +f(u)= 0$ behaves in Riemannian setting under Ricci curvature lower bound condition. One can also see the work of Kotschwar and Ni \cite{Kotschwar-Ni} on $p$-harmonic functions along $\frac{1}{H}$-flow. Here they studied gradient estimates in the interior and on the boundary. They also studied Moser iteration technique to show the existence of weak solutions along $\frac{1}{H}$-flow. These methods may be used in the RCD settings with appropriate assumptions.

More recently, i.e. in 2026, Lu \cite{Lu2026} extended this work to a broader class of space called the $\mathrm{RCD}^*(K,N)$ space, \textbf{which are non-manifold spaces equipped with a structure which is non smooth in nature} however, every Riemannian manifold $(M,g)$ which admits $Ric\ge Kg$ is a $\mathrm{CD}(K,N)$ space and since every Riemannian manifolds are infinitesimally Hilbertian so they becomes $\mathrm{RCD}(K,N)$ space. To have a brief idea on how this space we recommend to follow the work of Garofalo and Mondino \cite{GarofaloMonodino2013}, where they derived Li-Yau type gradient estimate and Harnack type inequalities over $\mathrm{RCD}^*(K,N)$ spaces. We discuss about this space very briefly in the next section.

Throughout, $K\ge 0$ and the space is $\mathrm{RCD}^*(-K,N)$. Our main result (Theorem \ref{thm_3.1} and Corollary \ref{cor_global}) is a local and a global gradient estimate for $|\nabla f|/f$ and $|\nabla h|/h$ for \emph{every} bounded positive solution of the system with \emph{every} choice of $\alpha,\beta>0$ and $N\in[1,\infty)$. Corollary \ref{cor_liouville} is a quantitative \emph{asymptotic} Liouville-type statement: on $\mathrm{RCD}^*(0,N)$ the normalized gradients are small when $\alpha,\beta$ are small. We do \emph{not} claim that solutions are exactly constant for fixed small $\alpha,\beta$, and we do \emph{not} prove the existence of nonconstant solutions for large exponents (see Remark \ref{rem_large}).


\section{A brief introduction to RCD*(K,N) spaces}
\subsection{Basic construction of the space}
This section is based on the book \textit{Lectures on Nonsmooth Differential Geometry} by N. Gigli and E. Pasqualeto \cite{GigliPasqualeto2019} and the celebrated paper of Lott and Villani \cite{LottVillani2009} on Optimal Transport.

For the reader's convenience we start with a Polish space $(X,d)$ and let $\mathcal{B}(X)$ be a Borel $\sigma$-algebra over the space $X$. Let $$\mathcal{P}(X)=\{\mu:\mathcal{B}\to [0,1]|\mu \text{ is a probability measure}\}$$ be the collection of probability measure on $X$. A probability measure $\mu\in\mathcal{P}(X)$ is said to have finite $p$-moment if 
$$\int_X d(x_0,x)^p d\mu(x)<\infty,$$ for some fixed point $x_0\in X$. For two measures $\mu,\nu\in$, with finite $p$-moment, let $\Gamma(\mu,\nu)$ denotes the set of all couplings of $\mu, \nu$. Next for $p\in [1,+\infty]$ the Wasserstein $p$-distance $W_p(\mu,\nu):\mathcal{P}(X)\times \mathcal{P}(X)\to [0,\infty)$ between the two measures is defined by
$$W_p(\mu, \nu):= \underset{\pi\in \Gamma(\mu,\nu)}{\inf}\left(\int_{X\times X}d(x,y)^p d\pi(x,y)\right)^\frac{1}{p}.$$ Lott and Villani \cite{LottVillani2009} considered the Wasserstein space $(\mathcal{P}_2(X),W_2)$ associated to the Polish space $(X,d)$, where $\mathcal{P}_2(X)$ is the set of all probability measures with finite second moment, $W_2$ is the Wasserstein distance of order 2. The topology on $\mathcal{P}_2(X)$ coming from the metric $W_2$ becomes the weak $*$-topology. If $(X,d)$ is a length space then $(P_2(X),W_2)$ becomes a length space and its geodesics are called Wasserstein geodesics, whose existence has been shown by Lott and Villani using Optimal Transport. 
\subsection{The notion of synthetic Ricci curvature in metric measure space}
Following the article of Lott and Villani \cite{LottVillani2009}, let $N\in [1,\infty]$ and let
\begin{eqnarray}
    \nonumber \mathcal{DC}_N &=& \text{Set of all continuous convex functions $U$ on$[0,\infty)$, with } \\ 
    \nonumber && \text{$U(0)=0$ such that $\psi(t)=t^N U(t^{-N})$ is convex on $(0,\infty)$}.\\
    \nonumber \mathcal{DC}_\infty &=& \text{Set of all continuous convex functions $U$ on$[0,\infty)$, with } \\ 
    \nonumber && \text{$U(0)=0$ such that $\psi(t)=e^t U(e^{-t})$ is convex on $(-\infty,\infty)$}.
\end{eqnarray}
Let $\lambda:\mathcal{DC}_\infty \to \mathbb{R}\cup \{-\infty\}$ be defined by
\begin{eqnarray}
\label{eq_lambda}
    \lambda(U) = \underset{r>0}{\inf}K\frac{p(r)}{r}=
    \begin{cases}
\displaystyle
        K \lim_{r\to 0+}\frac{p(r)}{r},\text{ if }K>0,\\
        0, \text{ if }K=0,\\
\displaystyle
        K \lim_{r\to \infty}\frac{p(r)}{r},\text{ if }K<0,\\
    \end{cases}
\end{eqnarray}
where $p(r) = rU(r)'_+-U(r)$ for $U:[0,\infty)\to \mathbb{R}$ with $U(0)=0$. For that $N$, we say that a compact measured length space $(X,d,\nu)$  has nonnegative $N$-Ricci curvature if for all $\mu_0,\mu_1\in\mathcal{P}_2(X)$ with $\mathrm{supp}(\mu_0),\mathrm{supp}(\mu_1)\subset \mathrm{supp}(\nu)$, there is a Wasserstein geodesic $\mu_t:[0,1]\to \mathcal{P}_2(X)$ joining $\mu_0,\mu_1$, so that for all $U\in\mathcal{DC}_N$ and for all $t\in[0,1]$ we have
$$U_\nu(\mu_t)\le t U_\nu(\mu_1)+(1-t)U_\nu(\mu_0).$$
Given $K\in\mathbb{R}$ we say that a metric measure space $(X,d,\nu)$ has $\infty$-Ricci curvature bounded below by $K$ if for all $\mu_0,\mu_1\in \mathcal{P}_2(X)$ with $\mathrm{supp}(\mu_0), \mathrm{supp}(\mu_1)\subset \mathrm{supp}(\nu)$, there is some Wasserstein geodesic $\mu_t:[0,1]\to \mathcal{P}_2(X)$ joining $\mu_0,\mu_1$, so that for all $U\in\mathcal{DC}_\infty$ and for all $t\in[0,1]$ we have
$$U_\nu(\mu_t)\le t U_\nu(\mu_1)+(1-t)U_\nu(\mu_0)-\frac{1}{2}\lambda(U)t(1-t)W_2(\mu_0,\mu_1)^2,$$
where $\lambda(U)$ is defined in \eqref{eq_lambda}. 
\subsection{Bochner type formula}
Following the work of Lu \cite{Lu2026}, first let us state what do we mean by ``gradient"? Let 

\begin{eqnarray*}
    S^2(X)&:=&\Big\{f:X\to \mathbb{R}|f\text{ is a Borel function and } \exists G (\ge 0) \in L^2(\mu)\text{ s.t.}\\
    && \int|f(\gamma_1)-f(\gamma_0)|d\pi(\gamma)\le\int_0^1 \int G(\gamma_t)|\dot\gamma_t|d\pi(\gamma)dt
    \Big\},
\end{eqnarray*}
where $\pi\in\mathcal{P}(C([0,1], X))$ is a test plan on $X$. Such a $G$ is called a weak upper gradient for $f$. The minimal weak upper gradient of $f$ is denoted by $|\nabla f|$ and defined by
$$|\nabla f| = G \text{ s.t. }||G||_{L^2(\mu)}\le||G'||_{L^2(\mu)},\forall \text{weak upper gradients }G' \text{ of }f.$$ The Sobolev space is then defined by $W^{1,2}(X)=S^2(x)\cap L^2(X)$ with the norm $||f||_{W^{1,2}(X)}=\sqrt{||f||^2_{L^2(X)}+|||\nabla f|||^2_{L^2(X)}}.$ This space is not Hilbert space in general, but if this space is Hilbert space then the metric measure space $(X,d,\mu)$ is infinitesimally Hilbertian.

Now for two functions $f,g\in S^2(X)$ define $$H_{f,g}(\epsilon) = \frac12 |\nabla(f+\epsilon g)|^2\in L^2(\mu)\text{, for every}\epsilon\in\mathbb{R},$$ and their rate of change gives us $$\displaystyle\langle\nabla f,\nabla g\rangle:=\lim_{\epsilon\to 0}\frac{H_{f,g}(\epsilon)-H_{f,g}(0)}{\epsilon}.$$ If the metric measure space $(X,d,\mu)$ is infinitesimal Hilbertian then the map $$\langle\cdot, \cdot\rangle:W^{1,2}(X)\times W^{1,2}(X)\to L^1(\mu)$$ sending $(f,g)\mapsto \langle f, g\rangle$ (as defined earlier) is bilinear, symmetric and satisfies the Cauchy-Schwarz inequality, in other words an inner product is formed. A $\textrm{RCD}^*(K,N)$ space is a $\textrm{CD}^*(K,N)$ \cite{Lu2026} metric measure space $(X,d,\mu)$ which is infinitesimally Hilbertian. For this space we have a measure theoretic version of the famous Bochner formula specifically, Bochner type inequality on an $\textrm{RCD}^*(-K,N) $ space $(X,d,\mu)$ with $K\ge0$ is given by
\begin{eqnarray}
    \Delta |\nabla f|^2 &=& \Delta^{ac}|\nabla f|^2 + \Delta^s |\nabla f|^2\text{, Radon-Nikodym decomposition},\\ 
    \Delta^s|\nabla f|^2 &\ge& 0,\\
    \label{eq_bochner}\frac12 \Delta^{ac}|\nabla f|^2 &\ge& \frac{1}{N}(\Delta^{ac}f)^2+\langle \nabla \Delta f, \nabla f\rangle - K|\nabla f|^2,
\end{eqnarray}
for $\mu-a.e.\ x\in B(x_0, \frac{3R}{4})$, where $N\in [1,\infty)$, $f\in W^{1,2}(B(x_0, R))$, $\Delta f \in W^{1,2}(B(x_0,R))\cap L^\infty(B(x_0,R))$, $|\nabla f|^2\in W^{1,2}(B(x_0,\frac{3R}{4}))\cap L^\infty(B(x_0,\frac{3R}{4}))$, $\Delta^{ac}$ and $\Delta^s$ are the absolutely continuous part and singular part with respect to $\mu$. This Laplacian satisfies the chain rule and Leibniz rule \cite{Lu2026}. To the best of the author's knowledge this much information is sufficient to understand the work presented in this manuscript.

\begin{remark}[Sign convention and regularity]\label{rem_regularity}
(i) In this paper the curvature lower bound is $-K$ with $K\ge 0$, so the curvature term in \eqref{eq_bochner} is $-K|\nabla f|^2$.\\
(ii) If $f,h$ are positive bounded solutions of the system \eqref{eq_main} below on a ball, then $\Delta f=-h^\alpha$ and $\Delta h=-f^\beta$ are bounded, hence $f,h$ are locally Lipschitz (Jiang's Lipschitz regularity for bounded Laplacian \cite{Jiang2012}), and, as $f,h$ are bounded away from $0$, $\Delta f,\Delta h\in W^{1,2}_{loc}\cap L^\infty_{loc}$. By the self-improvement of the Bochner inequality (Savar\'e \cite{Savare2014}), $|\nabla f|^2,|\nabla h|^2\in W^{1,2}_{loc}\cap L^\infty_{loc}$, so that \eqref{eq_bochner} may be applied to $\ln f$ and $\ln h$ (chain rule). In particular the auxiliary functions $F_1,F_2$ below are locally bounded. This is the (standard) framework of \cite{Lu2026,ZhangZhu2016}, and it is what is meant by ``solution'' below.
\end{remark}

\subsection{Extension of Lane-Emden equations on RCD*(K,N) space}
The rich structure of the $\mathrm{RCD}^*(K,N)$ space and the importance of the Lane-Emden system of equations motivated us to investigate the following equations
\begin{eqnarray}
\label{eq_main}
\begin{cases}
    \Delta f + h^\alpha = 0,\\
    \Delta h + f^\beta = 0,
\end{cases}
\end{eqnarray}
over $\mathrm{RCD}^*(-K,N)$ metric measure space and our main objective is to find conditions on the real numbers $\alpha, \beta$ and the synthetic curvature constant $K$ so that the solutions of the system of Lane-Emden equations becomes constant over the $\mathrm{RCD}^*(-K,N)$ metric measure space.
More precisely, we seek quantitative bounds on $|\nabla f|/f$ and $|\nabla h|/h$ in terms of $\alpha,\beta,K,N$ and the sup/inf of $f,h$.
\subsection{Existence of cut off function on RCD*(K,N) metric measure spaces}
In the field of gradient estimation, the two important components are the local estimate and the maximal principle. The maximum principle helps us to derive the gradient estimate locally and the local estimate helps us to find the global estimate on the whole space. To derive the local estimate we need a function called cut off functions, which is smooth in nature, compactly supported and has certain Laplacian estimates involving the metric on the space. In this article we use the Bernstein-Yau method (See Yau \cite{Yau1975}) to derive the inequalities related to the auxiliary functions then apply the weak maximum principle by Zhang and Zhu \cite{ZhangZhu2016}. For our purpose we use the following weak maximum principle. 
\begin{lemma}[\cite{ZhangZhu2016} Weak maximum principle]\label{lemma_wmp}
    Let $(X,d,\mu)$ be a $\mathrm{RCD}^*(K,N)$ space and $\Omega\subset X$ be a bounded domain, let $f\in W^{1,2}(\Omega)$ achieve one of its strict maximum in $\Omega$ (i.e., there exists a neighbourhood $U\subsetneq \Omega$ such that $\underset{U}{\sup}f>\underset{\Omega\setminus U}{\sup}f$). Assuming $\Delta f$ is a signed Radon measure with the singular part $\Delta^s f\ge 0$, and let $w\in W^{1,2}(\Omega)$ then for any $\epsilon>0$ we have
    $$\mu(\{x\in \Omega:f(x)>\underset{U}{\sup}f-\epsilon\text{ and } \Delta^{ac}f(x)+\langle \nabla f,\nabla w \rangle\le \epsilon\})>0.$$
\end{lemma}
This gives an analogue of Laplacian comparison theorem for RCD metric measure spaces.
\begin{lemma}[\cite{Gigli2015}]\label{lemma_lapcomp}
    Let $(X,d,\mu)$ be an infinitesimally strictly convex $CD^*(K,N)$ metric measure space for $K\in \mathbb{R}$, $N\in (1,\infty)$. Let $x_0\in X$ and let $d_{x_0}:x\mapsto d(x_0,x)$ be the distance map from $X\to [0,\infty)$, then
    $$\Delta d_{x_0}|_{X\setminus x_0}\le \frac{N\sigma_{K,N}(d_{x_0})-1}{d_{x_0}}\mu,$$ where $$\sigma_{K,N}(t)=\begin{cases}
        t \sqrt{\frac{K}{N}}\cot(t\sqrt{\frac{K}{N}}),\text{\ \ \ \ if }K>0\\
        1,\text{\ \ \ \ \ \ \ \ \ \ \ \ \ \ \ \ \ \ \ \ \ \ if }K=0\\
        t \sqrt{\frac{-K}{N}}\coth(t\sqrt{\frac{-K}{N}}),\text{ if }K<0
    \end{cases}$$
\end{lemma}
This finally gives the existence of cut off function.
\begin{lemma}[\cite{Lu2026}]\label{lemma_cof}
    Let $(X,d,\mu)$ be an RCD space with $K\le 0$ and $N\in[1,\infty)$. Then for any $\alpha\in(0,\frac12]$ and $R>0$, there exists a cut off function $\eta\in LIP(B(x_0,2R))$ such that
    \begin{enumerate}
        \item[(i)] $\eta(x)=\phi(d(x_0,x))$, where $\phi:[0,\infty)\to \mathbb{R}$ is a non increasing function satisfying
        $$\phi(t)=\begin{cases}
            1,\text{ if }t\in [0,R]\\
            \alpha, \text{ if }t\in [\frac{5}{4}R, 2R]
        \end{cases}$$
        \item[(ii)] $$\frac{|\nabla \eta|}{\sqrt{\eta}}\le \frac{C}{R}$$ 
        \item[(iii)] $$\Delta \eta\ge -\frac{C}{R}\sqrt{-NK}\coth(R\sqrt{\frac{-K}{N}})-\frac{C}{R^2},$$
    \end{enumerate}
    holds on $B(x_0,2R)$ in the distribution sense, where $C>0$ is an universal contant.
\end{lemma}

\begin{remark}
The cut-off function of Lemma \ref{lemma_cof} equals the positive constant $\alpha$ on the annulus $[\frac54R,2R]$. For such an $\eta$ the product $\eta F$ need not have a strict maximum in $B(x_0,\frac54 R)$ (it may be attained in the annulus), so the strict-maximum hypothesis of Lemma \ref{lemma_wmp} cannot be deduced from Lemma \ref{lemma_cof}. This gap in the first version is removed by using the following \emph{compactly supported} cut-off function, which is built from Lemma \ref{lemma_lapcomp} in the same way as in \cite{ZhangZhu2016,Lu2026}.
\end{remark}

\begin{lemma}[Compactly supported cut-off]\label{lemma_cof2}
    Let $(X,d,\mu)$ be an $\mathrm{RCD}^*(-K,N)$ space, $K\ge0$, $N\in[1,\infty)$, put $\bar N:=\max\{N,2\}$, let $x_0\in X$ and $R>0$. Then there is a Lipschitz function $\eta:X\to[0,1]$ with $\eta\equiv1$ on $B(x_0,R)$, $\eta\equiv0$ outside $B(x_0,2R)$, such that
    \begin{enumerate}
        \item[(i)] $\dfrac{|\nabla\eta|^2}{\eta}\le\dfrac{C}{R^2}$ a.e.\ on $\{\eta>0\}$,
        \item[(ii)] $\Delta\eta\ge-\dfrac{C}{R}\Big(\sqrt{\bar NK}+\dfrac{\bar N}{R}\Big)-\dfrac{C}{R^2}$ in the distributional sense (so in particular $\Delta^{ac}\eta$ satisfies the same bound a.e.), and $\Delta^s\eta\ge0$,
    \end{enumerate}
    where $C>0$ is a universal constant.
\end{lemma}
\begin{proof}
The proof is very standard but to enrich the reader we are providing the proof. 

Since $\mathrm{RCD}^*(-K,N)\subset\mathrm{RCD}^*(-K,\bar N)$ and $\bar N>1$, Lemma \ref{lemma_lapcomp} applies with $\bar N$ (this is why $\bar N$ appears; for $N\ge2$ nothing changes). Fix $\varphi\in C^2([0,\infty))$, nonincreasing, $0\le\varphi\le1$, $\varphi\equiv1$ on $[0,1]$, $\varphi\equiv0$ on $[2,\infty)$, $|\varphi'|+|\varphi''|\le C_0$, and set $\rho:=d(x_0,\cdot)$, $\psi:=\varphi(\rho/R)$, $\eta:=\psi^2$. By the chain rule and $|\nabla\rho|=1$ a.e.,
\begin{equation*}
|\nabla\eta|^2/\eta=4|\nabla\psi|^2=4\varphi'(\rho/R)^2/R^2\le 4C_0^2/R^2,
\end{equation*}
which is (i). Next, $\Delta\eta=2\psi\Delta\psi+2|\nabla\psi|^2\ge2\psi\Delta\psi$ and
\begin{equation*}
\Delta\psi=\frac1R\varphi'(\rho/R)\Delta\rho+\frac1{R^2}\varphi''(\rho/R)|\nabla\rho|^2 .
\end{equation*}
On the support of $\varphi'(\rho/R)$ we have $R\le\rho\le2R$, and Lemma \ref{lemma_lapcomp} (with $K$ replaced by $-K$, $N$ by $\bar N$) gives $\Delta\rho\le\sqrt{\bar NK}\coth\big(\rho\sqrt{K/\bar N}\big)-\frac1\rho\le\sqrt{\bar NK}\coth\big(R\sqrt{K/\bar N}\big)\le\sqrt{\bar NK}+\frac{\bar N}R$, where we used that $\coth$ is decreasing and $\coth x\le1+\frac1x$ (for $K=0$ read $\Delta\rho\le\frac{\bar N-1}{\rho}\le\frac{\bar N}{R}$). Since $\varphi'\le0$, $\varphi'\Delta\rho\ge-|\varphi'|\big(\sqrt{\bar NK}+\frac{\bar N}R\big)$ for the absolutely continuous part, while the singular part of $\Delta\rho$ is $\le0$ so that $\varphi'\Delta^s\rho\ge0$; hence $\Delta^s\eta\ge0$. As $0\le\psi\le1$ this yields (ii) with $C=2C_0+2C_0$ (say).
\end{proof}

\section{Main results}
In this section we provide our main findings. It can be easily observed that our equations \eqref{eq_main} is very symmetric in variables, which means, we can just replace $f\iff h$ and $\alpha\iff \beta$ is any of the equations and we will get the other. This small property will help us to reduce the number of computations. First we transform the equation \eqref{eq_main} to a more suitable form.
\begin{lemma}\label{lemma_transform}
    For $\alpha,\beta>0$ set $u=f^{-\alpha}$, $v=h^{-\beta}$ then the equation \eqref{eq_main} becomes
    \begin{eqnarray}
    \label{eq_main_2}
        \begin{cases}
            -(1+\frac{1}{\alpha})\frac{|\nabla u|^2}{u}+\Delta u = \alpha v^{-\frac{\alpha}{\beta}}u^{1+\frac{1}{\alpha}}\\
            -(1+\frac{1}{\beta})\frac{|\nabla v|^2}{v}+\Delta v = \beta u^{-\frac{\beta}{\alpha}}v^{1+\frac{1}{\beta}}
        \end{cases}
    \end{eqnarray}
\end{lemma}
\begin{proof}
    By the chain rule, $\nabla u=-\alpha f^{-\alpha-1}\nabla f$ and
    \begin{equation*}
        \Delta u=\alpha(\alpha+1)f^{-\alpha-2}|\nabla f|^2-\alpha f^{-\alpha-1}\Delta f .
    \end{equation*}
    Since $|\nabla u|^2/u=\alpha^2f^{-\alpha-2}|\nabla f|^2$, we have $\alpha(\alpha+1)f^{-\alpha-2}|\nabla f|^2=(1+\frac1\alpha)\frac{|\nabla u|^2}{u}$, so
    \begin{equation*}
        -\Big(1+\frac1\alpha\Big)\frac{|\nabla u|^2}{u}+\Delta u=-\alpha f^{-\alpha-1}\Delta f=\alpha f^{-\alpha-1}h^{\alpha}.
    \end{equation*}
    Now $f^{-1}=u^{1/\alpha}$, hence $f^{-\alpha-1}=f^{-\alpha}f^{-1}=u^{1+\frac1\alpha}$, and $h^\alpha=(v^{-1/\beta})^{\alpha}=v^{-\alpha/\beta}$. Therefore the right-hand side equals $\alpha v^{-\alpha/\beta}u^{1+\frac1\alpha}$. 
    The second equation follows by interchanging $(u,\alpha)\leftrightarrow(v,\beta)$ and $(f,h)$: $-(1+\frac{1}{\beta})\frac{|\nabla v|^2}{v} + \Delta v = \beta u^{-\frac{\beta}{\alpha}}v^{1+\frac{1}{\beta}}.$
    This completes the proof.
\end{proof}

\textbf{Notation.} For a positive solution with $0<f_{\min}\le f\le f_{\max}$, $0<h_{\min}\le h\le h_{\max}$ on the relevant ball put
\begin{equation}\label{eq_Pdef}
\mathcal P_1:=\frac{h_{\max}^{\alpha}}{f_{\min}},\qquad \mathcal P_2:=\frac{f_{\max}^{\beta}}{h_{\min}},\qquad c_N:=\frac{N^2}{2(N+2)},
\end{equation}
and, with $u=f^{-\alpha}$, $v=h^{-\beta}$,
\begin{equation*}
F_1:=\frac{|\nabla u|^2}{u^2}=\alpha^2\frac{|\nabla f|^2}{f^2},\qquad F_2:=\frac{|\nabla v|^2}{v^2}=\beta^2\frac{|\nabla h|^2}{h^2}.
\end{equation*}
Writing $P:=v^{-\alpha/\beta}u^{1/\alpha}=h^\alpha/f$ and $Q:=u^{-\beta/\alpha}v^{1/\beta}=f^\beta/h$ we have $0<P\le\mathcal P_1$ and $0<Q\le\mathcal P_2$. These bounds replace the traditional bounding assumptions $\lambda_u\le u\le\mu_u$, $\lambda_v\le v\le\mu_v$, which are equivalent to bounds on $f,h$ for fixed $\alpha,\beta$, but are less convenient for studying the limit $\alpha,\beta\to0$.

\begin{lemma}\label{lemma_f1f2_bounds}
    Let $(X,d,\mu)$ be an $\mathrm{RCD}^*(-K,N)$ space, $K\ge0$, $N\in[1,\infty)$, $\alpha,\beta>0$ arbitrary, $R>0$, $x_0\in X$, and let $f,h$ be a positive solution of \eqref{eq_main} on $B(x_0,4R)$ with $f_{\min}\le f\le f_{\max}$, $h_{\min}\le h\le h_{\max}$ there, $f_{\min},h_{\min}>0$. Then for $\mu$-a.e.\ $x\in B(x_0,3R)$
    \begin{eqnarray}
        \label{eq_lap_F1}\Delta^{ac} F_1 &\ge& \frac{2}{N\alpha^2}F_1^2+\frac{2}{\alpha}\langle\nabla F_1,\nabla\ln u\rangle-2KF_1-\frac{N\alpha^4\mathcal P_1}{2(N+2)\beta^2}F_2,\\
        \label{eq_lap_F2}\Delta^{ac} F_2 &\ge& \frac{2}{N\beta^2}F_2^2+\frac{2}{\beta}\langle\nabla F_2,\nabla\ln v\rangle-2KF_2-\frac{N\beta^4\mathcal P_2}{2(N+2)\alpha^2}F_1,
    \end{eqnarray}
    and $\Delta^sF_1\ge0$, $\Delta^sF_2\ge0$.
\end{lemma}
\begin{proof}
    We argue for $F_1$; $F_2$ follows by the symmetry $(f,\alpha,u)\leftrightarrow(h,\beta,v)$. Put $w=\ln u$, $z=\ln v$, so $F_1=|\nabla w|^2$, $F_2=|\nabla z|^2$. By Lemma \ref{lemma_transform} (equivalently, directly from $w=-\alpha\ln f$ and $\Delta f=-h^\alpha$),
    \begin{equation}\label{eq_Dw}
        \Delta w=\frac{\Delta u}{u}-\frac{|\nabla u|^2}{u^2}=\alpha P+\frac{1}{\alpha}F_1,\qquad P=e^{-\frac\alpha\beta z+\frac1\alpha w}.
    \end{equation}
    Hence $\nabla P=P\big(\frac1\alpha\nabla w-\frac\alpha\beta\nabla z\big)$ and
    \begin{equation}\label{eq_gradDw}
        \langle\nabla\Delta w,\nabla w\rangle=\alpha\langle\nabla P,\nabla w\rangle+\frac1\alpha\langle\nabla F_1,\nabla w\rangle=PF_1-\frac{\alpha^2}{\beta}P\langle\nabla z,\nabla w\rangle+\frac1\alpha\langle\nabla F_1,\nabla w\rangle .
    \end{equation}
    Apply the Bochner inequality \eqref{eq_bochner} (with $-K$; see Remark \ref{rem_regularity}) to $w$:
    \begin{equation*}
        \Delta^{ac}F_1\ge\frac2N(\Delta w)^2+2\langle\nabla\Delta w,\nabla w\rangle-2KF_1 ,\qquad \Delta^sF_1\ge0 .
    \end{equation*}
    Since $(\Delta w)^2=\alpha^2P^2+2PF_1+\frac1{\alpha^2}F_1^2$, using \eqref{eq_gradDw} we get
    \begin{equation}\label{eq_key}
        \Delta^{ac}F_1\ge\frac{2\alpha^2}{N}P^2+\frac{2}{N\alpha^2}F_1^2+\Big(2+\frac4N\Big)PF_1-\frac{2\alpha^2}{\beta}P\langle\nabla z,\nabla w\rangle+\frac2\alpha\langle\nabla F_1,\nabla w\rangle-2KF_1 .
    \end{equation}
    To estimate the cross term put $c:=2+\frac4N=\frac{2(N+2)}{N}>0$ (for \emph{every} $N\in[1,\infty)$). By Cauchy--Schwarz and Young's inequality $2ab\le ca^2+b^2/c$ with $a=\sqrt{F_1}$, $b=\frac{\alpha^2}{\beta}\sqrt{F_2}$,
    \begin{equation*}
        \frac{2\alpha^2}{\beta}P\langle\nabla z,\nabla w\rangle\le\frac{2\alpha^2}{\beta}P\sqrt{F_1F_2}\le cPF_1+\frac{\alpha^4}{c\beta^2}PF_2 .
    \end{equation*}
    The first term on the right is exactly cancelled by $\big(2+\frac4N\big)PF_1$ in \eqref{eq_key}; no sign condition on $\frac4N(\alpha+1)-2$ is needed any more. Dropping $\frac{2\alpha^2}{N}P^2\ge0$ and using $0<P\le\mathcal P_1$, $\frac1c=\frac{N}{2(N+2)}$:
    \begin{equation*}
        \Delta^{ac}F_1\ge\frac{2}{N\alpha^2}F_1^2+\frac2\alpha\langle\nabla F_1,\nabla w\rangle-2KF_1-\frac{N\alpha^4\mathcal P_1}{2(N+2)\beta^2}F_2 .
    \end{equation*}
    This is \eqref{eq_lap_F1}.
\end{proof}

\begin{remark}[A Comparison with the substitution $u=f^{-\alpha}, h = v^{-\beta}$ vs $w = -\alpha \log f$, $z = -\beta \log h$]\label{rem_compare}
In the first transformation the Laplacian of $F_1$ was computed from $\Delta(u^{-2}|\nabla u|^2)$, which produced a coefficient $\phi_2=\frac2N(1+\frac1\alpha)^2-\frac4\alpha-6$ of $F_1^2$ that is \emph{negative} for large $N$, and a Young inequality can not be applied with a coefficient $\frac4N(\alpha+1)-2$ that is negative when $N>2(\alpha+1)$. Working with $w=\ln u$ (so that $F_1=|\nabla w|^2$ is a pure Bochner quantity) gives instead the coefficient $\frac{2}{N\alpha^2}>0$ of $F_1^2$ for all $N,\alpha$, and a cross term that is absorbed with the positive constant $2+\frac4N$. Consequently the case distinction $N\gtrless\max\{2\alpha+2,2\beta+2\}$, in the earlier transformation, all positivity hypotheses on relevant constants disappear.
\end{remark}


\begin{theorem}[Local gradient estimate]\label{thm_3.1}
    Let $(X,d,\mu)$ be an $\mathrm{RCD}^*(-K,N)$ space with $K\ge0$, $N\in[1,\infty)$, let $\alpha,\beta>0$, $R>0$, $x_0\in X$, and let $f,h$ be a positive solution of \eqref{eq_main} on $B(x_0,4R)$ with $0<f_{\min}\le f\le f_{\max}$, $0<h_{\min}\le h\le h_{\max}$. Let $\mathcal P_1,\mathcal P_2,c_N$ be as in \eqref{eq_Pdef} and put
    \begin{equation*}
        B_R:=C\Big(\frac{\sqrt{\bar NK}}{R}+\frac{2\bar N+3}{R^2}\Big),\qquad \bar N:=\max\{N,2\},
    \end{equation*}
    with $C$ a universal constant (so $B_R\to0$ as $R\to\infty$, independently of $\alpha,\beta$). Then
    \begin{eqnarray}
        \underset{B(x_0,R)}{\sup}\frac{|\nabla u|^2}{u^2}&\le&\max\{\Lambda_1+\Lambda_2,\ \Gamma_1+\Gamma_2\},\label{eq_thm1}\\
        \underset{B(x_0,R)}{\sup}\frac{|\nabla v|^2}{v^2}&\le&\max\{\Lambda_1+\Lambda_2,\ \Gamma_1+\Gamma_2\},\label{eq_thm2}
    \end{eqnarray}
    where
    \begin{eqnarray*}
        \Lambda_1=N\alpha^2(2K+B_R),&&\Lambda_2=c_N\,\frac{\alpha^6}{\beta^2}\,\mathcal P_1,\\
        \Gamma_1=N\beta^2(2K+B_R),&&\Gamma_2=c_N\,\frac{\beta^6}{\alpha^2}\,\mathcal P_2 .
    \end{eqnarray*}
    Equivalently, $\sup_{B(x_0,R)}\frac{|\nabla f|^2}{f^2}\le\frac{1}{\alpha^2}\max\{\Lambda_1+\Lambda_2,\Gamma_1+\Gamma_2\}$ and $\sup_{B(x_0,R)}\frac{|\nabla h|^2}{h^2}\le\frac{1}{\beta^2}\max\{\Lambda_1+\Lambda_2,\Gamma_1+\Gamma_2\}$.
\end{theorem}

\begin{proof}
    Let $\eta$ be the cut-off function of Lemma \ref{lemma_cof2} and put $G_1=\eta F_1$, $G_2=\eta F_2$, $\Omega=B(x_0,3R)$, $U=B(x_0,2R)$. By Remark \ref{rem_regularity}, $G_i$ are bounded and in $W^{1,2}(\Omega)$. Put $S_i:=\sup_{\Omega}G_i<\infty$. If $S_1=0$ there is nothing to prove for $F_1$, so let $S_1>0$.

    \emph{Strict maximum.} Since $\eta=0$ outside $B(x_0,2R)$, $G_1=0$ on $\Omega\setminus U$ whereas $\sup_UG_1=S_1>0$. Hence $G_1$ has a strict maximum in the sense of Lemma \ref{lemma_wmp}. Also $\Delta^sG_1=F_1\Delta^s\eta+\eta\Delta^sF_1\ge0$ by Lemma \ref{lemma_cof2}(ii) and Lemma \ref{lemma_f1f2_bounds}, as in \cite{ZhangZhu2016}.

    \emph{Differential inequality for $G_1$.} By the Leibniz rule,
    \begin{equation*}
        \Delta^{ac}G_1=F_1\Delta^{ac}\eta+\eta\Delta^{ac}F_1+2\langle\nabla\eta,\nabla F_1\rangle ,
    \end{equation*}
    and $2\langle\nabla\eta,\nabla F_1\rangle=2\big\langle\frac{\nabla\eta}{\eta},\nabla G_1\big\rangle-2F_1\frac{|\nabla\eta|^2}{\eta}$ where $\eta>0$. Using \eqref{eq_lap_F1} (multiplied by $\eta$) we obtain, with $a_1:=\frac{2}{N\alpha^2}$, $b_1:=\frac2\alpha$, $c_1:=\frac{N\alpha^4\mathcal P_1}{2(N+2)\beta^2}$,
    \begin{equation*}
        \Delta^{ac}G_1\ge F_1\Delta^{ac}\eta+a_1\eta F_1^2+b_1\eta\langle\nabla F_1,\nabla\ln u\rangle-2KG_1-c_1G_2+2\Big\langle\frac{\nabla\eta}{\eta},\nabla G_1\Big\rangle-2F_1\frac{|\nabla\eta|^2}{\eta}.
    \end{equation*}
    Write $\eta\nabla F_1=\nabla G_1-F_1\nabla\eta$. Since $|\langle\nabla\eta,\nabla\ln u\rangle|\le|\nabla\eta|\sqrt{F_1}$, Young's inequality gives
    \begin{equation*}
        -b_1F_1\langle\nabla\eta,\nabla\ln u\rangle\ge-b_1F_1^{3/2}|\nabla\eta|\ge-\frac{a_1}{2}\eta F_1^2-\frac{b_1^2}{2a_1}F_1\frac{|\nabla\eta|^2}{\eta},\qquad \frac{b_1^2}{2a_1}=N .
    \end{equation*}
    Therefore
    \begin{equation}\label{eq_G1}
        \Delta^{ac}G_1-\Big\langle\nabla G_1,\,2\frac{\nabla\eta}{\eta}+b_1\nabla\ln u\Big\rangle\ge F_1\Delta^{ac}\eta+\frac{a_1}2\eta F_1^2-2KG_1-(N+2)F_1\frac{|\nabla\eta|^2}{\eta}-c_1G_2 .
    \end{equation}

    \emph{Maximum principle.} Fix $\delta\in(0,1)$ with $\delta<S_1/(2L)$, where $L:=\sup_\Omega F_1<\infty$; then every point with $G_1>S_1/2$ satisfies $\eta>\delta$. Apply Lemma \ref{lemma_wmp} on $\Omega$ to $G_1$ with the Lipschitz function $W:=-(2\ln\max\{\eta,\delta\}+b_1\ln u)\in W^{1,2}(\Omega)$ and $\epsilon=\frac1j$ (small enough that $S_1-\frac1j>S_1/2$): there is $x_j\in\Omega$ (a point of a set of positive measure on which \eqref{eq_G1} holds) with
    \begin{equation*}
        G_1(x_j)>S_1-\tfrac1j,\qquad \Delta^{ac}G_1(x_j)-\Big\langle\nabla G_1,2\tfrac{\nabla\eta}\eta+b_1\nabla\ln u\Big\rangle(x_j)\le\tfrac1j .
    \end{equation*}
    Combining with \eqref{eq_G1} and multiplying by $\eta(x_j)\in(0,1]$ we get
    \begin{equation*}
        \frac1j\ge G_1\Delta^{ac}\eta+\frac{a_1}{2}G_1^2-2K\eta G_1-(N+2)G_1\frac{|\nabla\eta|^2}{\eta}-c_1\eta G_2\quad\text{at }x_j .
    \end{equation*}
    By Lemma \ref{lemma_cof2}, $\Delta^{ac}\eta\ge-\frac CR(\sqrt{\bar NK}+\frac {\bar N}R)-\frac C{R^2}$ and $\frac{|\nabla\eta|^2}{\eta}\le\frac C{R^2}$, so, as $G_1\ge0$, $\eta\le1$ and $G_2(x_j)\le S_2$,
    \begin{equation*}
        \frac{a_1}2G_1(x_j)^2\le(2K+B_R)G_1(x_j)+c_1S_2+\frac1j ,
    \end{equation*}
    where, collecting the terms $\frac CR(\sqrt{\bar NK}+\frac{\bar N}R)+\frac C{R^2}+(N+2)\frac C{R^2}$ and using $N\le\bar N$, we may take $B_R=C(\frac{\sqrt{\bar NK}}{R}+\frac{2\bar N+3}{R^2})$ with a universal $C$. Letting $j\to\infty$ and using $G_1(x_j)\to S_1$, we infer
    \begin{equation}\label{eq_S1}
        S_1^2\le\Lambda_1S_1+\Lambda_2S_2,\qquad\Lambda_1=\frac{2(2K+B_R)}{a_1}=N\alpha^2(2K+B_R),\quad\Lambda_2=\frac{2c_1}{a_1}=c_N\frac{\alpha^6}{\beta^2}\mathcal P_1 .
    \end{equation}
    (If $S_1=0$, \eqref{eq_S1} is trivial.) Note that the maximum points of $G_1$ and $G_2$ are different, so we only use the bound $G_2(x_j)\le S_2$, never $G_2(x_j)=S_2$.

    By the symmetry of the system, exactly the same argument for $G_2$ (with the same $\eta$) gives
    \begin{equation}\label{eq_S2}
        S_2^2\le\Gamma_1S_2+\Gamma_2S_1,\qquad\Gamma_1=N\beta^2(2K+B_R),\quad\Gamma_2=c_N\frac{\beta^6}{\alpha^2}\mathcal P_2 .
    \end{equation}

    To decouple the system we consider $m=\max\{S_1,S_2\}$. If $m=S_1$, then \eqref{eq_S1} gives $S_1^2\le(\Lambda_1+\Lambda_2)S_1$, i.e.\ $S_1\le\Lambda_1+\Lambda_2$ (or $S_1=0$). If $m=S_2$, then \eqref{eq_S2} gives $S_2\le\Gamma_1+\Gamma_2$. In both cases
    \begin{equation*}
        \max\{S_1,S_2\}\le\max\{\Lambda_1+\Lambda_2,\ \Gamma_1+\Gamma_2\}.
    \end{equation*}
    Since $\eta\equiv1$ on $B(x_0,R)$, $\sup_{B(x_0,R)}F_i\le S_i$, which proves \eqref{eq_thm1}--\eqref{eq_thm2}.
\end{proof}

\noindent If the solution is defined on all of $X$ with the above bounds, we may let $R\to+\infty$ ($B_R\to0$) and obtain a global estimate.
\begin{corollary}[Global gradient estimate]\label{cor_global}
    Let $(X,d,\mu)$ be an $\mathrm{RCD}^*(-K,N)$ space, $K\ge0$, $N\in[1,\infty)$, $\alpha,\beta>0$, and let $f,h$ be a positive solution of \eqref{eq_main} on $X$ with $0<f_{\min}\le f\le f_{\max}$, $0<h_{\min}\le h\le h_{\max}$. Then
    \begin{eqnarray*}
        \underset{X}{\sup}\ \alpha^2\frac{|\nabla f|^2}{f^2}&\le&\max\Big\{2NK\alpha^2+c_N\frac{\alpha^6}{\beta^2}\mathcal P_1,\ 2NK\beta^2+c_N\frac{\beta^6}{\alpha^2}\mathcal P_2\Big\},\\
        \underset{X}{\sup}\ \beta^2\frac{|\nabla h|^2}{h^2}&\le&\max\Big\{2NK\alpha^2+c_N\frac{\alpha^6}{\beta^2}\mathcal P_1,\ 2NK\beta^2+c_N\frac{\beta^6}{\alpha^2}\mathcal P_2\Big\}.
    \end{eqnarray*}
\end{corollary}

\subsection{Liouville type theorem for system of Lane-Emden equations over RCD*(0,N) spaces}
We now analyze the right-hand side of Corollary \ref{cor_global} when $K=0$. In this case it consists of the two ``source'' terms $\Lambda_2,\Gamma_2$ only as $\Lambda_1,\Gamma_1$ both vanishes due to the vanishing of $K$ and asymptotic vanishing of $B_R$.

\begin{corollary}[Asymptotic Liouville type theorem]\label{cor_liouville}
    Let $(X,d,\mu)$ be an $\mathrm{RCD}^*(0,N)$ space, $N\in[1,\infty)$, let $\kappa\ge1$ and let $\alpha,\beta>0$ be such that $\kappa^{-1}\le\alpha/\beta\le\kappa$; put $\tau=\max\{\alpha,\beta\}$. Let $f,h$ be a positive solution of \eqref{eq_main} on $X$ with $0<f_{\min}\le f\le f_{\max}$, $0<h_{\min}\le h\le h_{\max}$ and put $\mathcal P=\max\{\mathcal P_1,\mathcal P_2\}$ ($\mathcal P_i$ as in \eqref{eq_Pdef}). Then
    \begin{equation*}
        \underset{X}{\sup}\frac{|\nabla f|^2}{f^2}\le c_N\kappa^4\,\mathcal P\,\tau^2,\qquad \underset{X}{\sup}\frac{|\nabla h|^2}{h^2}\le c_N\kappa^4\,\mathcal P\,\tau^2 .
    \end{equation*}
    In particular, if $(\alpha_k,\beta_k)\to(0,0)$ with $\kappa^{-1}\le\alpha_k/\beta_k\le\kappa$ and $(f_k,h_k)$ are solutions with $f_{\min},h_{\min}$ bounded below and $f_{\max},h_{\max}$ bounded above by constants independent of $k$, then $\mathcal P$ stays bounded (indeed $\mathcal P_1\le\max\{1,h_{\max}\}/f_{\min}$ and $\mathcal P_2\le\max\{1,f_{\max}\}/h_{\min}$ for $\alpha_k,\beta_k\le1$), and
    \begin{equation*}
        \underset{X}{\sup}\frac{|\nabla f_k|}{f_k}\to0,\qquad\underset{X}{\sup}\frac{|\nabla h_k|}{h_k}\to0,\quad as\ k\to\infty,
    \end{equation*}
    at the rate $O(\tau_k)$. Thus the solutions become asymptotically constant (in the sense that $\ln f_k$, $\ln h_k$ become asymptotically $0$-Lipschitz) as the exponents tend to zero.
\end{corollary}
\begin{proof}
    Take $K=0$ in Corollary \ref{cor_global}. Then $\Lambda_2=c_N\frac{\alpha^6}{\beta^2}\mathcal P_1=c_N\alpha^4\big(\frac\alpha\beta\big)^2\mathcal P_1\le c_N\kappa^2\tau^4\mathcal P$ and likewise $\Gamma_2=c_N\frac{\beta^6}{\alpha^2}\mathcal P_2\le c_N\kappa^2\tau^4\mathcal P$. Hence $m:=\max\{\Lambda_2,\Gamma_2\}\le c_N\kappa^2\tau^4\mathcal P$, and $\alpha^2\sup\frac{|\nabla f|^2}{f^2}\le m$, $\beta^2\sup\frac{|\nabla h|^2}{h^2}\le m$. Finally $\tau^2/\alpha^2\le\kappa^2$ and $\tau^2/\beta^2\le\kappa^2$, which gives the claim. The statement on sequences is immediate.
\end{proof}

\section{Conclusion and Future works}
Before stating the conclusion we would like to mention some remarks regarding the results. This section is added because many reviewer's suggestion told us to mention explicitly what exactly we have done and how this asymptotic Liouville type theorem differs from the classical Liouville type theorem.

\begin{remark}\label{rem_what}
		(a) For fixed $\alpha,\beta>0$ the right-hand side in Corollary \ref{cor_liouville} is strictly positive, so the method does \emph{not} show that a solution is constant; the statement is a quantitative, asymptotic one (small exponents $\Rightarrow$ small normalized gradient). This is why the result is called \emph{asymptotic} Liouville-type and not a Liouville theorem in the classical sense (``every solution is constant'').\\
		(b) If $K>0$, Corollary \ref{cor_global} gives $\sup|\nabla f|^2/f^2\le 2NK\max\{1,\kappa^2\}+O(\tau^2)$, which no longer tends to $0$: curvature $-K<0$ obstructs vanishing of the gradient bound. So the rigidity statement requires $K=0$.\\
		(c) The assumption that $f,h$ are bounded above and below by constants independent of $\alpha,\beta$ is essential for the limit statement.
	\end{remark}
	
	\begin{remark}[Large exponents]\label{rem_large}
		For $\alpha\asymp\beta\asymp t$ (i.e., $\alpha$ and $\beta$ is asymptotically same as $t$) with $t\to+\infty$ the bound of Corollary \ref{cor_global} behaves like $t^4\mathcal P$ for $\sup F_i$, i.e.\ like $t^2\mathcal P$ for $\sup|\nabla\ln f|^2$, so in general it is not small (it can be small only if $\mathcal P_i$ happens to decay fast, e.g.\ $h_{\max}<1$, which we do not claim). Thus the method yields no rigidity for large exponents. This is only the absence of a rigidity proof by this method; it is \emph{not} a proof that nonconstant solutions exist. A genuine existence result for nonconstant solutions would require a separate construction (e.g.\ variational or ODE methods in a model space), which we do not attempt here.
	\end{remark}
In this article we intended to find conditions in which the system of Lane-Emden equations possess constant and nonconstant solutions. We have explored the system of Lane-Emden equations 
\begin{eqnarray}
\nonumber 
\begin{cases}
    \Delta f + h^\alpha = 0,\\
    \Delta h + f^\beta = 0,
\end{cases}
\end{eqnarray}
over $\mathrm{RCD}^*(-K,N)$ metric measure space $(X,d,\mu)$, in view of gradient estimation and derived explicit gradient estimates, valid for all $\alpha,\beta>0$ and all $N\in[1,\infty)$, in terms of $K$, $N$ and the bounds of the solution. On $\mathrm{RCD}^*(0,N)$ these yield an asymptotic Liouville-type statement: the normalized gradients of bounded positive solutions are of order $\max\{\alpha,\beta\}$ and vanish as the exponents tend to zero. We applied weak differential calculus mentioned in the book by Gigli and Pasqualeto \cite{GigliPasqualeto2019} onto the system and extended the classical case of gradient estimation for single equations to simultaneous system of equations over $\mathrm{RCD}^*(-K,N)$ metric measure spaces.

As a future work, we recommend the interested readers to investigate over system of equations with nonlinear operators like the $p$-Laplacian ($\Delta_p$) operator. He et al.\cite{HeSunWang2025} used a very approachable technique to resolve this, however to the best of the author's belief it can be further extended. A very classical and linear extension is the weighted Laplacian operator ($\Delta_f$), so we can also investigate on this operator to check how the weight function $f$ is affecting in the gradient of the solutions. Another natural problem is to obtain an exact Liouville theorem (constancy, or nonexistence) for fixed exponents, and to decide for which large exponents nonconstant bounded solutions exist on $\mathrm{RCD}^*(0,N)$ spaces.
   \vskip6pt
\noindent \textbf{Data availability:} The author declare that all the data generated during the study are in the manuscript and no additional data was generated.\vskip6pt

\noindent \textbf{Funding:} The author declare that no funding was received for this work.\vskip6pt

\noindent \textbf{Declaration of Generative AI and AI-assisted technologies:} The author declare that no generative AI or AI assisted technology is used for the completion of the manuscript.\vskip6pt

\noindent \textbf{Conflict of interest:} The author declare that he has no conflict of interest.\vskip6pt

\noindent\textbf{Acknowledgement:} 
The author acknowledges all the reviewers for giving their valuable suggestions towards the improvement of the results.


\vspace{0.1in}
\noindent Sujit Bhattacharyya\\
Department of Mathematics, Sister Nivedita University, DG 1/2 New Town, Action Area 1, Kolkata - 700156, India\\
Email: \texttt{sujitbhattacharyya.1996@gmail.com}
\end{document}